\documentclass{amsart} 
\author{Andy Hammerlindl}
\title{On expanding foliations}

\usepackage{amsmath}
\usepackage{amssymb}
\usepackage{amsfonts}
\usepackage{amsthm}
\usepackage{graphicx}

\newcommand{\R}{\mathbb{R}}
\newcommand{\Z}{\mathbb{Z}}
\newcommand{\N}{\mathbb{N}}
\newcommand{\T}{\mathbb{T}}
\newcommand{\bbT}{\mathbb{T}}
\newcommand{\bbS}{\mathbb{S}}

\newcommand{\Es}{E^s}
\newcommand{\Ec}{E^c}
\newcommand{\Eu}{E^u}

\newcommand{\Ws}{W^s}

\newcommand{\Wu}{W^u}

\newcommand{\inv}{^{-1}}

\newcommand{\tilf}{\tilde f}
\newcommand{\tilg}{\tilde g}
\newcommand{\tilp}{\tilde p}

\newcommand{\tilM}{\tilde M}
\newcommand{\tilN}{\tilde N}
\newcommand{\tilW}{\tilde W}
\newcommand{\tilTM}{\widetilde {T_1M}}
\newcommand{\tilphi}{\tilde{\varphi}}

\newcommand{\diam}{\operatorname{diam}}
\newcommand{\rad}{\operatorname{rad}}
\newcommand{\id}{\operatorname{id}}
\newcommand{\Diff}{\operatorname{Diff}}
\newcommand{\PH}{\operatorname{PH}}

\newtheorem{thm}{Theorem}[section]
\newtheorem{cor}[thm]{Corollary}
\newtheorem{lemma}[thm]{Lemma}
\newtheorem{prop}[thm]{Proposition}

\theoremstyle{remark}
\newtheorem*{remark} {\bf Remark}

\newtheorem*{notation} {\bf Notation}

\providecommand{\acknowledgement}{{\noindent\bf Acknowledgements}\quad}

\begin{document}

\begin{abstract}
    Certain families of manifolds which support Anosov flows do not support
    expanding, quasi-isometric foliations.
\end{abstract}
\maketitle


\section{Introduction}
This paper demonstrates that certain manifolds do not admit foliations which are
both expanding and whose leaves satisfy a form of quasi-isometry.
That is, if $M$ belongs to one of several families of manifolds listed in
the theorems below, it is impossible to find a diffeomorphism $f:M \to M$
and a foliation $W$ such that

\begin{itemize}
    \item
    $W$ is \emph{invariant}: $f(W)=W$,

    \item
    $W$ is \emph{expanding}: there is $\lambda > 1$ such that
    $\|Tf v\|  \ge  \lambda \|v\|$ for all $v \in T W$,


    \item
    $W$ is \emph{quasi-isometric}:
    letting $\tilW$ denote the lift of $W$ to the universal cover $\tilM$, 
    there is a global constant $Q > 1$ such that
    $d_{\tilW}(x,y) < Q\ d_{\tilM}(x,y) + Q$
    for all $x$ and $y$ on the same leaf of $\tilW$.
\end{itemize}
A major motivation for investigating expanding, quasi-isometric foliations is
the study of partially hyperbolic systems, diffeomorphisms of the form $f:M \to
M$ with an invariant splitting $TM=\Eu \oplus \Ec \oplus \Es$ such that the
\emph{unstable} $\Eu$ subbundle is expanding under $Tf$, the \emph{stable} $\Es$
is contracting, and the \emph{center} $\Ec$ neither expands as much as $\Eu$ nor
contracts as much as $\Es$.
In general, partially hyperbolic systems are difficult to analyze and
classify.  In the case where the foliations $\Wu$ and $\Ws$ tangent to $\Eu$ and
$\Es$ are quasi-isometric, the situation is much improved.  Under such an
assumption, the center subbundle $\Ec$ is uniquely integrable \cite{Brin}, which
is not true in general \cite{BW-dc}.  Moreover, the system enjoys a form of
structural stability \cite{Hammerlindl}.
Any partially hyperbolic system on the 3-torus must have quasi-isometric
invariant foliations \cite{BBI2}, and this has been used to give a
classification for these systems \cite{ham-thesis}.  Both the establishment of
quasi-isometry and the resulting classification can be extended to 3-manifolds
with nilpotent fundamental group \cite{Parwani,ham-nil}.
Further results hold in higher dimensions \cite{ham-conseq,ham-pgps}.

In light of the results cited above, a natural approach to analyze partially
hyberbolic systems on a given manifold is to first establish quasi-isometry of
the invariant foliations, and then use this to prove further properties of the
system.  This paper shows that for many manifolds supporting partially
hyperbolic diffeomorphisms, this approach is impossible.

\begin{thm} \label{main}
    A closed manifold does not support an expanding quasi-isometric foliation
    if it is:
    \begin{enumerate}
        \item
        a $d$-dimensional Riemannian manifold of constant negative curvature
        where $d  \ge  3$,
        \item
        the unit tangent bundle of
        a $d$-dimensional Riemannian manifold of constant negative curvature
        where $d  \ge  3$, or
        \item
        the suspension of a hyperbolic toral automorphism.
    \end{enumerate}  \end{thm}
Many examples of partially hyperbolic systems come from the time-one maps of
Anosov flows, and 
a classic example of an Anosov flow is the geodesic flow on a
negatively curved manifold $M$. This flow is defined on the unit tangent
bundle $T_1M$ as in case (2) above.
Another example of an Anosov flow is the suspension of an Anosov
diffeomorphism.  If the diffeomorphism is defined on a torus $\bbT^d$, it
corresponds to case (3).
It is conjectured that every codimension one Anosov flow in dimension
$d  \ge  4$ is of this form \cite{ghys-codim}.
Note that Theorem \ref{main} is not specific to the case of foliations coming
from Anosov flows.
In fact, it is easy to show that no Anosov flow (on any manifold) can
have a quasi-isometric strong stable or unstable foliation.

In his original paper on the subject, Fenley showed that certain manifolds do
not permit quasi-isometric codimension one foliations \cite{Fenley}.  This
paper considers foliations of any codimension with the additional condition of
expanding dynamics.  This extra condition is needed as in cases (2) and (3),
the orbits of the Anosov flows mentioned above give one-dimensional
quasi-isometric foliations.


\medskip

The proof of Theorem \ref{main} relies on analyzing the fundamental group of
the manifold, and the following generalization holds.

\begin{thm} \label{fundy}
    A closed manifold does not support an expanding quasi-isometric foliation
    if its fundamental group is isomorphic to the fundamental group of a
    manifold listed in Theorem \ref{main}.
\end{thm}
The proof involves Mostow Rigidity and the techniques could be easily applied
to more general locally symmetric spaces.  For the benefit of those
dynamicists not well-versed in geometric group theory, this paper only
treats the specific case of hyperbolic manifolds.

As suggested by Ali Tahzibi, one could also consider \emph{non-uniformly}
expanding foliations and similar results hold under additional assumptions.
For the benefit of those geometers not well-versed in non-uniform
hyperbolicity, this discussion is left to the appendix.
\section{Preliminaries}
\begin{notation}
    A \emph{lift} of a function $f:M \to N$ is a choice of
    function $\tilf:\tilM \to \tilN$ such that $P_N \tilf = f P_M$ where $P_M:\tilM
    \to M$ and $P_N:\tilN \to N$ are the universal coverings.
    Viewing the fundamental group as the set of deck transformations on $\tilM$,
    $\tilf$ uniquely determines a group homomorphism $f_*: \pi_1(M) \to \pi_1(N)$
    which satifies $f_*(\alpha)\tilf(x)=\tilf(\alpha(x))$ for $x \in \tilM$ and
    $\alpha \in \pi_1(M)$.

\end{notation}
For a foliation to be expanding as defined above, we require that
the function $f:M \to M$ is $C^1$ and that each leaf of the foliation is
$C^1$ as a submanifold.
The
foliation itself need only be continuous, as is commonly the case for
foliations encountered when studying dynamical systems.
Also, since the proofs of Theorems \ref{main} and \ref{fundy} do not use the
fact that a foliation covers the entire manifold, the results also hold for
laminations in place of foliations.

The following is an immediate consequence of the definitions of expanding and
quasi-isometric.

\begin{lemma}
    If the foliation $W$ is quasi-isometric and expanding under $f:M \to M$
    then for a lift $\tilf:\tilM \to \tilM$ and distinct points $x$ and $y$ on the
    same leaf of the lifted foliation $\tilW$, the sequence
    $\{d_{\tilM}(\tilf^n(x),\tilf^n(y))\}$ grows exponentially.
\end{lemma}
If we can establish that for any homeomorphism $f:M \to M$ with lift
$\tilf:\tilM \to \tilM$ and any $x,y \in \tilM$, the sequence
$\{d(\tilf^n(x),\tilf^n(y))\}$ grows subexponentially, then there can be no
expanding quasi-isometric foliation on $M$.  This is the technique used to
prove Theorem \ref{main}.

\begin{lemma} \label{bdd}
    Let $M$ and $N$ be manifolds, $M$ be compact, and $f,g:M \to N$ be
    continuous functions with lifts $\tilf, \tilg: \tilM \to \tilN$ such that the
    induced homomorphisms $f_*,g_*:\pi_1(M) \to \pi_1(N)$ are equal.  Then, there
    is $C>0$ such that $d_{\tilN}(\tilf(x), \tilg(x)) < C$ for all $x \in \tilM$.
\end{lemma}
\begin{proof}
    The function $\tilM \to \R,\, x \mapsto d_{\tilN}(\tilf(x), \tilg(x))$ is
    invariant under deck transformations.  It descends to a function $M
    \to \R$ and is therefore bounded.
      \end{proof}
\begin{cor}
    If a foliation $W$ is quasi-isometric and expanding under $f:M \to M$ then
    the induced homomorphism $f_*$ is not equal to the identity.
\end{cor}
\begin{cor}
    No time-one map of an Anosov flow or perturbation thereof has a
    quasi-isometric strong stable or unstable foliation.
\end{cor}
Theorem \ref{main} follows from Theorem \ref{fundy}.  However, since cases (1) and
(2) of Theorem \ref{main} have short, direct proofs, we give them first for
illustrative purposes.

\begin{prop}
    Let $M$ be a compact manifold of constant negative curvature,
    $\dim M  \ge  3$, and $f:M \to M$ a homeomorphism with lift
    $\tilf:\tilM \to \tilM$.
    Then, there is $C>0$ such that
    $d(\tilf^n(x), \tilf^n(y)) < d(x,y) + C n$ for $x,y \in \tilM$.
\end{prop}
\begin{proof}
    By Mostow rigidity, there is an isometry $g:M \to M$ and lift $\tilde
    g:\tilde M \to \tilde M$ such that $f_*=g_*$ as
    automorphisms of $\pi_1(M)$.  By Lemma \ref{bdd}, for $x,y \in \tilM$,
    \begin{align*}
        d(\tilf(x), \tilf(y))
        & \le d(\tilf(x), \tilg(x))+d(\tilg(x),\tilg(y))+d(\tilg(y),\tilf(y)) \\
        & \le C + d(x,y) + C
      \end{align*}
    and the claim follows by induction.
\end{proof}
\begin{prop}
    Let $M$ be a compact manifold of constant negative curvature,
    $\dim M  \ge  3$, and let $T_1M$ be the unit tangent bundle.
    If $f:T_1M \to T_1M$ is a homeomorphism with lift $\tilf:\tilTM \to \tilTM$,
    there is $C>0$ such that $d(\tilf^n(x), \tilf^n(y)) < d(x,y) + C n$ for
    $x,y \in \tilTM$.
\end{prop}
\begin{proof}
    The unit tangent bundle $T_1M$ fibers over $M$ with fiber $\bbS^k$, $k>2$.
    The long exact sequence of homotopy groups for a fibration
    \[
        \ldots \to \pi_1(\bbS^k) \to \pi_1(T_1M) \to \pi_1(M) \to \pi_0(\bbS^k) \to \ldots
    \]
    shows that the projection $p:T_1M \to M$ induces an isomorphism $p_*$ on the
    fundamental groups.  By Mostow rigidity, there is an isometry $g:M
    \to M$ such that $p_*f_*p_*\inv=g_*$.  After lifting,
    $d_{\tilM}(\tilp(\tilf(x)), \tilg(\tilp(x)))$ is bounded for $x
    \in \tilTM$.  Arguing as in the last proof, for any $x$ and $y$
    \[
        d(\tilp(\tilf(x)), \tilp(\tilf(y)))
        < d(\tilp(x), \tilp(y)) + C
    \]
    so that
    \[
        d(\tilp(\tilf^n(x)), \tilp(\tilf^n(y)))
        < d(\tilp(x), \tilp(y)) + n C,
    \]
    and as $p_*$ is an isomorphism, one can show that there is a global
    constant $R>1$ such that
    $
        d_{\tilM}(\tilp(x),\tilp(y)) <
        R d_{\tilTM}(x,y) + R.
    $
    From these inequalities the proof follows.
\end{proof}
\section{The general proof}
To prove Theorem \ref{fundy} and case (3) of Theorem \ref{main}, we reason more
abstractly.
Suppose $M$ is a compact manifold with universal covering $\tilM$,
and $f:M \to M$ is a diffeomorphism with lift $\tilf:\tilM \to \tilM$,
which induces an automorphism 
$f_*: \pi_1(M) \to \pi_1(M)$.

Fix a fundamental domain $K \subset \tilM$ and for a subset $A \subset \pi_1(M)$
define $AK = \{\alpha x : \alpha \in A, x \in K\}$.  Observe that $\tilf(AK) =
f_*(A)\tilf(K)$ and if $A'$ is another subset of $\pi_1(M)$, then
$A A' K = (A A') K = A(A' K)$ is well-defined.

Fix a finite set of generators for $\pi_1(M)$ and define a metric on the group
by word distance.
There is a constant $C>0$ such that $d_{\tilM}(\alpha_i x, x) < C$ for every
generator $\alpha_i$ of $\pi_1(M)$ and all $x \in \tilM$.  Consequently, for a
subset $A \subset \pi_1(M)$,
\[
    \diam(AK)  \le  C \diam(A) + \diam(K)
\]
where the diameters of $AK$ and $K$ are measured on $\tilM$ and
$\diam(A)$ is with respect to the word metric.

As $\tilf(K)$ is compact, there is an integer $N$ such
that $\tilf(K) \subset B_N K$
where $B_N = \{ \alpha \in \pi_1(M) : |\alpha|  \le  N \}$.
The word metric is defined such that the $N$-neighbourhood $U_N(A)$ of a set
$A \subset \pi_1(M)$ is given by $A B_N$, and therefore
\[
    \tilf(AK) = f_*(A) \tilf(K) \subset f_*(A) B_N K = U_N(f_*(A)) K.
\]
Starting with a subset $A_0 \subset \pi_1(M)$, define a sequence $\{A_k\}$ by
$A_{k+1} = U_N(f_*(A_k))$.
One can prove by induction that $\tilf^k(A_0K) \subset A_k K$ for all $k \ge 1$.
If the diameter of $A_k$ grows at most polynomially, then the diameter of
$\tilf^n(A_0K)$ does as well.
The above reasoning is summed up in the following proposition.

\begin{prop} \label{growpoly}
    Suppose $G$ is a finitely generated group with the following property:

    \begin{quote}
        For every automorphism $\phi:G \to G$, integer $N>0$ and starting set
        $A_0 \subset G$, the sequence $\{A_k\}$ defined by $A_{k+1} =
        U_N(\phi(A_k))$ grows at most polynomially in diameter.
    \end{quote}
    Then, for any manifold $M$ with $\pi_1(M) = G$, diffeomorphism $f:M \to M$
    with lift $\tilf: \tilM \to \tilM$ and bounded subset $K \subset \tilM$,
    the diameter of $\tilf^n(K)$ grows at most polynomially as $n \to \infty$.
      \end{prop}
\begin{notation}
    For lack of a better word, call any group $G$ satisfying the hypothesis of
    Proposition \ref{growpoly} \emph{unstrechable}.
\end{notation}
\begin{cor}
    There is no expanding quasi-isometric foliation on a manifold with
    unstretchable fundamental group.
\end{cor}
We consider the fundamental groups of hyperbolic manifolds at the end of this
section.
For now, consider the fundamental group arising from a manifold included in case
(3) of Theorem \ref{main}.

\begin{prop} \label{toralstretch}
    The fundamental group of a suspension of a hyperbolic toral automorphism
    is unstretchable.
\end{prop}
To prove this proposition, consider $\pi_1(M)$ as an abstract group $G$.
It fits into a exact sequence
\[
    0 \to \Z^d \to G \to \Z \to 0.
\]
Let $H \triangleleft G$ be the image of $\Z^d$ in this sequence and fix an
element $z \in G$ such that its image under the projection $G \to \Z$ generates
$\Z$.  Every element of $G$ may then be written uniquely as $x \cdot z^k$ where
$x \in H$ and $k \in \Z$.  Further, there is an automorphism $A:H \to H$, coming
from the hyperbolic toral automorphism, such that $z \cdot x = (Ax) \cdot z$ for
all $x \in H$.

\begin{lemma} \label{lemma-autG}
    The automorphisms of $G$ are exactly those of the form $\phi(x) = Bx$ for
    $x \in H$ and $\phi(z)=v \cdot z^e$ where $B \in Aut(H) \approx GL(d,Z)$, $v \in
    H$, $e = \pm 1$, and $A^e B = B A$.
\end{lemma}
This result is well known, at least in the case $d = 2$.
For completeness, we give a short proof for general $d$,
starting with the following claim.

\begin{lemma} \label{lemma-charH}
    $H$ is a characteristic subgroup:
    if $\phi$ is an automorphism of $G$, then $\phi(H) = H$.
\end{lemma}
\begin{proof}
    We will show that $H = \rad([G,G])$, that is, $v \in H$ if and only
    if there is $k \in \Z$ such that $v^k$ is in $[G,G]$.  As this is a purely
    group-theoretic characterization, it is preserved under isomorphism.
    Note that the image of a commutator $u v u \inv v \inv$ under a map
    $G \to \Z$ must be zero.  By the above short exact sequence,
    $[G,G] < H$ and $\rad([G,G]) < H$ as well.

    To show the other inclusion, note that for $x \in H$,
    \[    
        [x,z] = x \cdot z \cdot (-x) \cdot z \inv
        = (x - Ax) \in H
    \]
    and therefore $(A-I)H \subset [G,G]$ where $I:H \to H$ denotes the identity.
    Taking $A-I$ to be an $n \times n$ matrix, if $(A-I)\Z^d$ did not have full
    rank, it would mean $A-I$ has a nullspace (in both $\Z^d$ and $\R^d$), but
    $A$ is hyperbolic, implying that $A-I$ is invertible over $\R^d$.
    Therefore, $(A-I)\Z^d$ has full rank, and $\rad((A-I)\Z^d) = \Z^d$.
    Consequently, $H = \rad((A-I)H) < \rad([G,G])$.
\end{proof}
\begin{proof}
    [Proof of Lemma \ref{lemma-autG}]
    Let $\phi:G \to G$ be an automorphism.  From Lemma
    \ref{lemma-charH},
    $\phi(H)=H$, so define $B := \phi|_H \in Aut(H)$.  Further, $\phi$
    induces an automorphism on the quotient $G/H \approx \Z$ which must be of
    the form $\pm \id:\Z \to \Z$.  Therefore, the coset $z H$ maps to the coset
    $z^{\pm 1}H$ which is the case exactly when $\phi(z) = v \cdot z^{\pm 1}$ for
    some $v \in H$.
    To be well-defined, $\phi$ must satisfy $\phi(z) \cdot \phi(x) = \phi(Ax) \cdot
    \phi(z)$ for all $x \in H$.  This is equivalent to the condition
    $A^e B = B A$.  The converse direction is straightforward to verify.
\end{proof}
Now fix $\phi \in Aut(G)$, and define $b$, $v$, and $e$ as in Lemma
\ref{lemma-autG}.
For simplicity, assume $e = 1$. The case with $e=-1$ is similar.
Define a metric $\|\cdot\|$ on $H \approx \Z^d \subset \R^d$ using the standard
metric on $\R^d$.  Fix a very large positive constant $\lambda$, and define for
$\ell, h \in \N$ the set
\[
    B(\ell, h) = \{ x \cdot z^k \in G : \|x\|  \le  \lambda^\ell, |k|  \le  h \}.
\]
If $x_1, \ldots, x_m$ is a list of all elements of $H$ with norm one, then $\{
x_1, \ldots, x_m, z, z \inv \}$ is a generating set for $G$.
This determines a word metric on $G$.

\begin{lemma}
    For all $\ell, h  \ge  1$,
    \[
        U_1(B(\ell,h)) \subset B(\ell+h, h+1)
    \]
    and for $N  \ge  1$,
    \[
        U_N(B(\ell,h)) \subset B(\ell + N(h+N), h+N).
    \]  \end{lemma}
\begin{proof}
    Suppose $x \cdot z^k \in B(\ell,h)$.  Then if $y \in H$ is a generator,
    $\|y\|=1$ and
    $
        (x \cdot z^k) \cdot y = (x + A^k y) \cdot z^k.
    $
    As $A$ is fixed, we may assume $\lambda$ was chosen large enough that
    $\|A^{\pm 1} u\| < \lambda u$ for all $u \in H$.  Then (assuming also
    $\lambda>2$),
    \[
        \|x + A^k y\| < \|x\| + \lambda^k \|y\|  \le  \lambda^\ell + \lambda^h
         \le  \lambda^{\ell+h}
    \]
    proving $(x \cdot z^k) \cdot y \in B(\ell+h, h+1)$.  The case
    $(x \cdot z^k) \cdot z^{\pm 1} = x \cdot z^{k \pm 1}$ is immediate.
    The second half of the lemma is proved by induction using
    $U_{n+1}(A) = U_1(U_n(A))$.
\end{proof}
\begin{lemma}
    For $\ell,h  \ge  2$, \quad $\phi(B(\ell,h)) \subset B(\ell+h, h+1)$.
\end{lemma}
\begin{proof}
    Recall $\phi$ is defined by $\phi(x)=Bx$ for $x \in H$ and $\phi(z)=v \cdot z$.
    Then for $k>0$,
    $\phi(z^k) = (v \cdot z)^k = (\sum_{i=0}^{k-1} A^i v) \cdot z^k$
    as can be proved by induction.
    As $v$ is fixed, we may assume $\lambda$ was chosen large enough that
    $\|v\| + \|A v\| < 1 + \lambda$ and $\|A^i v\| < \lambda^i$ for all $i>2$.
    These conditions imply
    $\|\sum_{i=0}^{k-1} A^i v\| < \sum_{i=0}^{k-1} \lambda^i < \lambda^k.$
    Also, assume $\|Bx\| < \lambda x$ for all $x \in H$.
    If $x \cdot z^k \in B(\ell,h)$ with $k  \ge  0$, then
    \[
        \phi(x \cdot z^k) = (Bx + \sum_{i=0}^{k-1} A^i v) \cdot z^k
    \]
    where
    \[
        \|Bx + \sum A^i v\|  \le  \lambda \|x\| + \lambda^k
         \le  \lambda^{\ell+1} + \lambda^h
         \le  \lambda^{\ell+h}
    \]
    so $\phi(x \cdot z^k) \in B(\ell+h, h)$.
    The case of $x \cdot z^k$ with $k$ negative follows by the same reasoning
    with $A \inv$ in place of $A$.
\end{proof}
\begin{remark}
    We assumed $h \ge 2$ above so that $\lambda^{\ell+1} + \lambda^h  \le 
    \lambda^{\ell+h}$ would hold.
\end{remark}
Now, as in the hypothesis of Proposition \ref{growpoly}, assume $N$ is fixed, and
$A_0$ is a finite subset of $G$ which defines a sequence $\{A_k\}$ by
$A_{k+1}=U_N(\phi(A_k))$.  As $A_0$ is finite, it is contained in some
$B(\ell,h)$ for large enough $\ell$ and $h$.
Then,
\begin{align*}
    A_1
    &\subset U_N(\phi(B(\ell,h))) \\
    &\subset U_N(B(\ell+h,h)) \\
    &\subset B(\ell+h+ N(h+N), h+N) \\
    &\subset B(\ell + 2(h+N)^2, h+N).
\end{align*}
By induction,
$A_k \subset B(\ell + p(k), h + N k)$
where $p(k) := \sum_{i=1}^{k} 2(h + N i)^2$ grows at most polynomially in $k$.
To show $\diam(A_k)$ is growing polynomially, it is enough to show that the
diameter of $B(\ell,h)$ is polynomial in $\ell$ and $h$.  In fact, the
dependence is linear.

\begin{lemma} \label{polylh}
    There is $C>0$ such that $\diam(B(\ell,h)) < C \ell+h$.
\end{lemma}
\begin{proof}
    To prove this, we move our study from the group back to the manifold.
    The automorphism $A$ on $H \approx \Z^d$ may be thought of as a
    hyperbolic toral automorphism defining the manifold
    $M = \T^d \times \R / \sim$ under the relation $(x,t+1) \sim (Ax, t)$.
    Define a Riemannian metric on $M$ such that the submanifold $\T^d \times
    \{0\}$ is equipped with the usual flat metric on $\T^d = \R^d / \Z^d$ and
    such that the flow on $M$ defined by $\varphi_t(x,s) = (x,s+t)$ flows at
    unit speed.

    Lift the metric to the universal cover $\tilM = \R^d \times \R$.  The lifted
    flow $\tilphi_t(x,s)=(x,s+t)$ is Anosov and the strong stable manifold
    through the origin, $\Ws(0,0)$, is a linear subspace of $\R^d \times \{0\}$
    corresponding to the stable manifold of $A$.  Suppose $(x,0)$ is a point
    on $\Ws(0,0)$ with $\|x\|>1$.
    Then, there is $\sigma>0$ such that
    \[
        d_s(\tilde \varphi_t(x,0), \tilde \varphi_t(0,0))  \le 
        e^{-\sigma t} d_s((x,0), (0,0))
    \]
    for all $t>0$ where $d_s$ is distance measured along the strong
    stable leaf.
    Choosing $t$ such that $e^{-\sigma t}\|x\| = 1$, one can show that
    $d((x,0), (0,0))  \le  \tfrac{2}{\sigma} \log \|x\| + 1$.
    A similar formula holds for points on the same unstable leaf.  As the
    stable and unstable foliations are linear and transverse, it follows that
    there is a constant $C>0$ such that 
    $d((x,0),(0,0))  \le  C \log \|x\|$ for any $x \in \R^d$ with $\|x\| > 1$.

    The embedding $i:G \to \tilM$, $x \cdot z^k \mapsto (x,k)$ agrees with the
    standard method of embedding a fundamental group in the universal cover.
    In particular, it
    is a quasi-isometric function and there is $Q>1$ such that
    %
    %
    \[    
        d_G(x,0) < Q\, d_{\tilM}(i(x), i(0)) < C Q \log \|x\|
    \]
    for all non-zero $x \in H$.
    From here it is straightforward to show that
    $
        d_G(x \cdot z^k, 0) < C Q \log(\lambda) \ell + h
    $
    for all $x \cdot z^k \in B(\ell,h)$, completing the proof.
\end{proof}
Lemma \ref{polylh}, with the discussion preceeding it, concludes the proof of
Proposition \ref{toralstretch}.
We have shown that a manifold constructed as the suspension of a
hyperbolic toral automorphism does not have an expanding quasi-isometric
foliation.  This completes case (3) of Theorem \ref{main} and also part of
Theorem \ref{fundy}, since the application of Proposition \ref{growpoly} depends
only on the group $\pi_1(M)$ and not the manifold itself.  To finish the proof
of Theorem \ref{fundy},
we consider groups coming from hyperbolic manifolds.


\begin{prop}
    The fundamental group of a $d$-dimensional manifold of
    constant negative curvature ($d  \ge  3$) is unstretchable.
      \end{prop}
\begin{proof}
    Let $G$ be such a group, and
    let $\phi$, $N$, and the sequence $\{A_k\}$ be as in Proposition
    \ref{growpoly}.
    For a subset $A \subset G$, note that
    $
        \phi(U_N(A)) \subset U_C(\phi(A))
    $
    where $C = \max \{ |\phi(x)| : x \in G, |x|  \le  N \}$.  Therefore, for any
    $p>0$,
    $
        A_{k+p} \subset U_{N'} (\phi^p(A_k))
    $
    for some integer $N'$ depending on $p$, $N$, and $C$.

    As a consequence of Mostow rigidity, the group of outer automorphisms,
    $Out(G)$, is finite (see remark below).
    Hence, there
    is $p$ such that $\phi^p$ is an inner automorphism
    $x \mapsto g \inv x g$.
    For $x,y \in G$,
    \[
        d(\phi(x),\phi(y)) = |g \inv x \inv g   g \inv y g|
         \le  |g| + |x \inv y| + |g|
        = d(x,y) + 2 |g|,
    \]
    Thus, $A_{k+p} \subset U_{N'} (U_{2|g|} (A_k) )$ from which it follows that
    $\{A_k\}$ grows at most polynomially.
\end{proof}
\begin{remark}
    The fact that $Out(G)$ is finite is well known to those studying
    rigidity.
    However, I was unable to find a citable elementary proof.
    For readers not familiar with the result,
    I give an outline of the proof here.

    As $M$ is aspherical, an automorphism $\phi$ of $\pi_1(M, x_0)$ is induced
    by a homotopy equivalence $h:(M,x_0) \to (M,x_0)$. By Mostow rigidity, $h$
    is homotopic to an isometry $g:(M,x_0) \to (M,x_0)$.  As this homotopy
    does not preserve the base point $x_0$, the automorphisms $\phi=h_*$ and
    $g_*$ are conjugate, but not necessarily identical.
    Now choose paths $\alpha_i$ ($i \in \{1,\ldots,n\}$) which represent the
    generators of $\pi_1(M,x_0)$.
    For each $i$, the path $g \circ \alpha_i$ is the same length as $\alpha_i$
    and so there is a finite number of possibilities for the element of
    $\pi_1(M, x_0)$ which it represents.  Hence, there are only a finite number
    of possibilities for $g_*$.
\end{remark}

\appendix 
\section{Non-uniform expansion}

Suppose a diffeomorphism $f:M \to M$ has an invariant, one-dimensional foliation
$W$.
By Oseledets theorem, there is a full probability set $R \subset M$ such that
for $x \in R$, the Lyapunov exponent
\[
    \lambda^W(x) := \lim_{|n| \to \infty} \frac{1}{n} \log \|Tf^n|_{TW(x)}\|
\]
exists \cite{Oseledets-proof}.

\begin{prop} \label{zeroset}
    Suppose $f:M \to M$ is a diffeomorphism of a manifold with unstretchable
    fundamental group, and $W$ is an invariant quasi-isometric foliation.
    Then,
    \[
        R' := \{ x \in R : \lambda^W(x)  \ne  0 \}
    \]
    intersects each leaf of $W$ in a set of (one-dimensional) Lebesgue measure
    zero.

    Moreover, if $W$ is absolutely continuous, then $R'$ has Lebesgue
    measure zero as a subset of $M$.
\end{prop}
\begin{remark}
    There are several possible ways to define absolute continuity
    (see, for example, \S 2.6 of \cite{regis-thesis}).
    Here, we take absolute continuity of a foliation to mean
    that any set $X$ which intersects each leaf in a null set, is itself a
    null set on $M$.  Then, the second half of the proposition follows
    immediately from the first half.
\end{remark}
\begin{proof}
    The proof is an adaptation of an idea explained in
    \cite[Proposition 0.5]{baraviera2003}.  There, it is originally
    attributed to Ma\~n\'e.

    Assume the proposition is false for some $f$ and $W$.
    By replacing $f$ with $f \inv$ if necessary,
    we may assume there is a constant $c > 0$ and a precompact subset $A$ of a
    leaf $L$ of $W$ such that $A$ has positive Lebesgue measure and
    $\lambda^W(x) > c$ for all $x \in A$.

    For a positive
    integer $k$, let $A_k$ denote the set of all points $x \in A$ such that
    $
        \frac{1}{n} \log \|Tf^n|_{TW(x)}\| > c
    $
    for all $n > k$.  As $\bigcup A_k = A$, there is $k$ such that $A_k$ has
    positive Lebesgue measure as a subset of $L$.  Further, the Lebesgue
    measure of $f^n(A_k)$ grows exponentially fast.  By quasi-isometry, the
    diameter of $A$ (as a subset of $\tilde M$) grows exponentially fast,
    contradicting Proposition \ref{growpoly}.
\end{proof}
In several cases, non-zero Lyapunov exponents have been used to show that the
center foliations of partially hyperbolic systems are not absolutely
continuous \cite{ShubWilkinson, baraviera2003}.  We show that the same
technique applies here.

Let $m$ be a measure equivalent to Lebesgue on a compact manifold $M$.
Let $\Diff^1_m(M)$ denote all $C^1$ diffeomorphisms on $M$ which preserve $m$,
and
let the subset $\PH^1_m(M)$
denote partially hyperbolic diffeomorphisms with one-dimensional center.
$\PH^1_m(M)$ is open with respect to the $C^1$ topology on
$\Diff^1_m(M)$.

If $f \in \PH^1_m(M)$ has a center foliation $W^c_f$ which satisfies a technical
condition known as \emph{plaque expansiveness},
it follows that
there is a neighbourhood
$U$ of $f$ such that every $g \in U$ also has a center foliation $W^c_g$.
Moreover, the foliations are equivalent; there is a homeomorphism $h$
(depending on $g$) taking leaves of $W^c_f$ to those of $W^c_g$.
Plaque expansiveness can be established in many specific cases, and it is an
open question if all center foliations are plaque expansive (see
\cite{Hammerlindl}, \cite{RHRHU-survey}, and \cite{HPS} for more details).

\begin{prop}
    Suppose $M$ has unstretchable fundamental group, and $f \in \PH^1_m(M)$.
    Further, suppose $W^c_f$ exists and is plaque expansive and
    quasi-isometric.
    Then, for an open and dense set of $g$ close to $f$, $W^c_g$ is not
    absolutely continuous.
    To be precise, there are open subsets $U,V \subset \PH^1_m(M)$ such that
    $f \in U \subset \bar V$ and $W^c_g$ is not absolutely continuous
    for all $g \in V$.
      \end{prop}
\begin{proof}
    Let $U$ be the open neighbourhood of $f$ given by plaque expansiveness
    \cite{HPS}.  In particular, there exists a foliation $W^c_g$ tangent to the
    center direction $E^c_g$ for every $g \in U$.  As this foliation is
    equivalent to $W^c_f$, it is also quasi-isometric.
    For $g \in U$, define
    \[
        \lambda^c(g) := \int_M \log \|Tg|_{E^c_g(x)}\| d m(x)
    \]
    and $V := \{ g \in U : \lambda^c(g)  \ne  0 \}$.
    As the function $g \mapsto \lambda^c(g)$ is continuous, $V$ is open.
    It follows from \cite[Proposition 0.3]{baraviera2003} that $V$ is dense in
    $U$.
    Suppose $g \in V$.
    By the Birkhoff ergodic theorem,
    \[
        \lambda^c(x) := \lim_{|n| \to \infty} \frac{1}{n} \log \|Tg^n|_{E^c_g(x)}\|
    \]
    is defined almost everywhere and 
    $
        \int_M \lambda^c(x) = \lambda^c(g).
    $
    Therefore, $\lambda^c(x)$ is non-zero on a positive measure set, and by
    Proposition \ref{zeroset}, $W_g^c$ is not absolutely continuous.
\end{proof}
\bigskip

\acknowledgement
The author thanks Alex Eskin, Benson Farb, Ali Tahzibi, Charles Pugh, and Amie
Wilkinson for helpful conversations.


\bibliographystyle{plain}
\bibliography{dynamics}

\end{document}